\documentclass{article}%
\usepackage{amsmath}
\usepackage{graphicx}
\usepackage{amscd}%
\usepackage{amsfonts}%
\usepackage{amssymb}
\newtheorem{theorem}{Theorem}[section]
\newtheorem{corollary}[theorem]{Corollary}

\newtheorem{example}[theorem]{Example}
\newtheorem{lemma}[theorem]{Lemma}
\newtheorem{proposition}[theorem]{Proposition}
\newtheorem{remark}[theorem]{Remark}
\newenvironment{proof}[1][Proof]{\textbf{#1.} }{\ \rule{0.5em}{0.5em}}

\begin{document}

\title{Breuil-Kisin Modules and Hopf Orders in Cyclic Group Rings}
\author{Alan Koch\\Agnes Scott College}
\maketitle

\begin{abstract}
For $K$ a finite extension of $\mathbb{Q}_{p}$ with ring of integers $R$ we
show how Breuil-Kisin modules can be used to determine Hopf orders in $K$-Hopf
algebras of $p$-power dimension. We find all cyclic Breuil-Kisin modules, and
use them to compute all of the Hopf orders in the group ring $K\Gamma$ where
$\Gamma$ is cyclic of order $p$ or $p^{2}.$ We also give a Laurent series
interpretation of the Breuil-Kisin modules that give these Hopf orders.

\end{abstract}

Let $R$ be a complete discrete valuation ring of mixed characteristic $\left(
0,p\right)  $ with quotient field $K$ and perfect residue field $k.$ Let
$e=e\left(  K/\mathbb{Q}_{p}\right)  $ be the absolute ramification index.
Then $R$ is an extension of $W:=W\left(  k\right)  ,$ the ring of Witt vectors
with coefficients in $k$, and we may write $K=K_{0}\left[  x\right]  /\left(
E\left(  x\right)  \right)  ,$ where $K_{0}=$Frac$\left(  W\right)  $ and
$E\left(  x\right)  $ is an Eisenstein polynomial. For any finite group
$\Gamma$ the group rings $R\Gamma$ and $K\Gamma$ have the structure of Hopf
algebras over $R$ and $K$ respectively. Clearly we have $R\Gamma\subset
K\Gamma,$ a relationship we can express by extension of scalars:
$R\Gamma\otimes_{R}K\cong K\Gamma.$ However, if $\Gamma$ is a $p$-group then
the $R$-Hopf algebra $R\Gamma$ is not uniquely determined by this isomorphism,
i.e. there exist other finitely generated projective $R$-Hopf algebras
$H\subset K$ such that $H\otimes_{R}K\cong K\Gamma.$ Such an $H$ must
necessarily contain $R\Gamma$ \cite[5.2]{Childs00}, and $H$ is called an
$R$-Hopf order in $K\Gamma.$ More generally, given a $K$-Hopf algebra $A,$ an
$R$-Hopf order is a submodule of $A$ which is an $R$-Hopf algebra such that
the extension of scalars from $R$ to $K$ produces an isomorphism with $A$.

The classification of Hopf orders is useful in local Galois module theory as
it helps in solving the normal integral basis problem \cite{ChildsMoss94},
particularly Hopf orders in group rings. If $p$ does not divide the order of
the group $\Gamma$ then the answer is simple: the only $R$-Hopf order in
$K\Gamma$ is $R\Gamma$ \cite[20.3]{Childs00}$.$ Thus we consider the cases
where $\left|  \Gamma\right|  =p^{n}$ for some $n>0.$ Two examples of such
groups (and the only examples for $n\leq2$) are $\Gamma=C_{p^{n}}$ and
$\Gamma=\left(  C_{p}\right)  ^{n},$ where $C_{m}$ is the cyclic group of
order $m$. The former case, particularly for $n\leq3,$ has been studied by
Byott, Childs, Greither, Underwood et al \cite{Byott93a} \cite{Byott93b}
\cite{ChildsUnderwood04} \cite{Greither92} \cite{Underwood94}
\cite{Underwood96} \cite{UnderwoodChilds06}. In the latter, Childs, Greither,
and Smith \cite{GreitherChilds98} \cite{ChildsSmith05} have found some Hopf
orders in certain circumstances.

Our contribution in the elementary abelian case was the introduction of Breuil
modules to the problem. Using Breuil's theory of ``filtered free'' modules
\cite{Breuil00} we get a categorical equivalence between $R$-Hopf algebras
such that the endomorphism $\left[  p\right]  :=($mult$)^{p}\circ\left(
\Delta\otimes1\right)  ^{p}$ is trivial, and a collection of Breuil modules,
specifically free $k\left[  u\right]  /\left(  u^{pe}\right)  $-modules
$\mathcal{M}$ with a $k\left[  u\right]  /\left(  u^{pe}\right)  $-submodule
$\mathcal{M}_{1}\subset u^{e}\mathcal{M}$ and a map $\phi_{1}:\mathcal{M}%
_{1}\rightarrow\mathcal{M}$ satisfying certain properties. In \cite{Koch07a}%
\ it was shown that the Hopf algebra corresponding to a rank $n$ Breuil module
$\mathcal{M}$ was a Hopf order in $\left(  KC_{p}\right)  ^{n}$ if and only if
$\mathcal{M}_{1}$ contained $n$ elements $x_{1},\dots,x_{n}\;$which are
$\mathbb{F}_{p}$-linearly independent and such that $\phi_{1}\left(
x_{i}\right)  =x_{i}$ for all $i$.

Unfortunately, Breuil modules do not easily generalize to the case where
$\Gamma$ is not elementary abelian. The ring $k\left[  u\right]  /\left(
u^{pe}\right)  $ must be replaced with the $p$-adic completion of the divided
power envelope of $W\left[  u\right]  $ with respect to the ideal $\left(
E\left(  u\right)  \right)  .$ Furthermore, the submodule $\mathcal{M}_{1}$
must contain the $p$-adic completion of the ideal generated by $\left(
E\left(  u\right)  \right)  ^{i}/i!$ for all $i\geq1$, making these modules
hard to classify. Another drawback to using Breuil module theory for Hopf
orders is that the category of Breuil modules is not abelian, and while every
Hopf order $H$ in $K\Gamma$ comes with an injection $R\Gamma\hookrightarrow
H,$ the corresponding monomorphism of Breuil modules is generally not
one-to-one. While in \cite{Koch07a}\ we found a criterion for a map
$\mathcal{M\rightarrow M}^{\prime}$ of Breuil modules of the same rank to be
one-to-one, in general this appears to be a difficult problem.

In this paper, we show how one can find Hopf orders in Hopf algebras of rank
$p^{n}$ not necessarily killed by $\left[  p\right]  $ using a related theory
we call Breuil-Kisin modules. Along with the theory above, Breuil
\cite{Breuil98} conjectured another correspondence between a certain category
of modules and finite flat commutative group schemes (and hence finite
projective commutative, cocommutative Hopf algebras).\ The conjecture, with a
slight modification, was later proved by Kisin \cite{Kisin06}, and provides a
simplification to the theory above. A Breuil-Kisin module is a module over
$W\left[  \left[  u\right]  \right]  $ with a map $\phi$ satisfying certain
properties.\ These properties depend on the Eisenstein polynomial, but in a
simpler manner than for Breuil modules. An indication of the usefulness of
Breuil-Kisin theory is the following fact: two Breuil-Kisin modules
$\mathfrak{M}$ and $\mathfrak{M}^{\prime}$ correspond to generically
isomorphic $K$-Hopf algebras if and only if $\mathfrak{M}\left[
u^{-1}\right]  \cong\mathfrak{M}^{\prime}\left[  u^{-1}\right]  ,$ i.e.
$\mathfrak{M}$ and $\mathfrak{M}^{\prime}$ become isomorphic over the ring of
Laurent series $W\left(  \left(  u\right)  \right)  .$

We start with a review of the Breuil-Kisin theory, its connection to Breuil
modules as well as to $R$-Hopf algebras. We explain how Hopf orders are easier
to identify with Breuil-Kisin modules than with Breuil modules. Then, we look
at the simplest class of Breuil-Kisin modules: the ones which are generated
over $W\left[  \left[  u\right]  \right]  $ by a single element. We refer to
this collection as the class of cyclic Breuil-Kisin modules. Cyclic
Breuil-Kisin modules can be used to classify all Hopf orders in $KC_{p},$ but
no nontrivial orders in $KC_{p^{n}}$ for $n\geq2$ can be found this way. (By
``nontrivial orders'' we mean orders which are not $RC_{p^{n}},$ although in
general a finite $K$-Hopf algebra need not have any $R$-Hopf orders -- see
\cite[20.5]{Childs00}.) However, in the following section we construct all
Hopf orders in $KC_{p^{2}}$ using extensions of cyclic Breuil-Kisin modules.
Finally, we show how it is possible (although at this point perhaps not
practical) to describe Hopf orders in $KC_{p^{n}}$ by picking $n$ elements in
$W_{n}\left(  \left(  u\right)  \right)  $ which satisfy certain properties.

We have chosen to use Breuil-Kisin theory to find Hopf orders in $KC_{p^{n}},$
but it can be used to find Hopf orders in any abelian (commutative,
cocommutative) finite projective $K$-Hopf algebra of $p$-power rank. In fact,
it appears that finding orders in the dual $\left(  KC_{p^{n}}\right)  ^{\ast
}$ may be slightly easier. Thus if $K$ contains a primitive $\left(
p^{n}\right)  ^{\text{th}}$ root of unity $\zeta_{n}$ then we could develop
the theory by considering only $\left(  KC_{p^{n}}\right)  ^{\ast}$ since it
is isomorphic to $KC_{p^{n}}.$ However, we do not insist that $\zeta_{n}\in
K$, in contrast to many of the works cited above.

Breuil-Kisin modules have been used predominantly to study Galois
representations. We hope that the results below will encourage their use in
other Hopf algebra applications. For example, in \cite{Koch01b} and
\cite{Koch05} we classified monogenic Hopf algebras (i.e. Hopf algebras
generated by a single element) over discrete valuation rings with $e=1$ and
$e\leq p-1$ respectively. In \cite{Koch09} we classified monogenic Hopf
algebras regardless of ramification, however we were limited to such Hopf
algebras $H$ where Spec$\left(  H\right)  $ was killed by $p$. This latter
paper used Breuil modules. Monogenic Hopf algebras are important in the study
of Hopf-Galois extensions (see, e.g. \cite{Koch03}). It seems likely that
Breuil-Kisin modules could be used to obtain a more general classification.

Throughout this paper, $R,K,K_{0},E,W,$ $e$, and $p$ are as above. Also, let
$c_{0}=E\left(  0\right)  /p,$ and let $F\left(  u\right)  =\left(  E\left(
u\right)  -u^{e}\right)  /p.$ All group schemes are affine, commutative, flat,
and have order a power of $p$; likewise all Hopf algebras are finite,
projective, abelian, and of $p$-power rank. The author would like to thank
Dajano Tossici for his suggestions on Theorem \ref{p2}.

\section{Breuil-Kisin Modules}

Here we review the basic theory of Breuil-Kisin modules. More details can be
found in \cite{Kisin06} and \cite{Kisin09}.

Let $\mathfrak{S}=W\left[  \left[  u\right]  \right]  $ and $\mathfrak{S}%
_{n}=\mathfrak{S}/p^{n}\mathfrak{S}=W_{n}\left[  \left[  u\right]  \right]  .$
Let $\phi:\mathfrak{S\rightarrow S}$ be the Frobenius-semilinear (hereafter
``semilinear'') map extending the Frobenius on $W$ such that $\phi\left(
u^{i}\right)  =u^{pi}$ for all $i\geq0.$ Clearly $\phi\left(  p^{n}%
\mathfrak{S}\right)  \subset p^{n}\mathfrak{S},$ so we have induced semilinear
maps on $\mathfrak{S}_{n}$ for all $n$ which we also denote by $\phi.$ For an
$\mathfrak{S}$-module $\mathfrak{M}$ we define $\mathfrak{S}\otimes_{\phi
}\mathfrak{M}$ to be $\mathfrak{S}\otimes_{\mathfrak{S}}\mathfrak{M}$, where
$\mathfrak{S}$ is viewed as an $\mathfrak{S}$-module via $\phi.$ Explicitly,
\[
\phi\left(  s\right)  s^{\prime}\otimes_{\phi}m=s^{\prime}\otimes_{\phi}sm
\]
for all $s,s^{\prime}\in\mathfrak{S}$ and $m\in\mathfrak{M}.$ In particular,
$1\otimes_{\phi}um=u^{p}\otimes_{\phi}m$ for $m\in\mathfrak{M}.$ This
$\mathfrak{S}$-module is denoted by both $\mathfrak{S}\otimes_{\phi
,\mathfrak{S}}\mathfrak{M}$ and $\phi^{\ast}\left(  \mathfrak{M}\right)  $ in
\cite{Kisin09}.

We define the category $^{\prime}$(Mod/$\mathfrak{S}$) to be the category of
$\mathfrak{S}$-modules $\mathfrak{M}$ together with a semilinear map
$\phi_{\mathfrak{M}}:\mathfrak{M\rightarrow M}$ such that $E\mathfrak{M}$
annihilates $\mathfrak{M}/\left(  1\otimes_{\mathfrak{S}}\phi_{\mathfrak{M}%
}\right)  \left(  \mathfrak{S\otimes}_{\phi}\mathfrak{M}\right)  .$ Here we
view $\left(  1\otimes_{\mathfrak{S}}\phi_{\mathfrak{M}}\right)  \left(
\mathfrak{S\otimes}_{\phi}\mathfrak{M}\right)  $ as a submodule of
$\mathfrak{M}$ via the canoncial isomorphism $\mathfrak{S\otimes
}_{\mathfrak{S}}\mathfrak{M\rightarrow M}$: note that $\left(  1\otimes
_{\mathfrak{S}}\phi_{\mathfrak{M}}\right)  $ maps $\mathfrak{S\otimes}_{\phi
}\mathfrak{M}$ to $\mathfrak{S\otimes}_{\mathfrak{S}}\mathfrak{M}$. Thus we
require that for all $m\in\mathfrak{M}$ there exist $\left\{  m_{i}\right\}
\subset\mathfrak{M}$ and $\left\{  s_{i}\right\}  \subset\mathfrak{S}$ such
that%
\[
Em=\sum s_{i}\phi_{\mathfrak{M}}\left(  m_{i}\right)  .
\]
As expected, an $\mathfrak{S}$-linear map $\alpha:\mathfrak{M\rightarrow
M}^{\prime}$ is a morphism in this category if $\alpha\phi_{\mathfrak{M}}%
=\phi_{\mathfrak{M}^{\prime}}\alpha.$

From here on we follow convention and write $\phi$ for $\phi_{\mathfrak{M}},$
$\phi_{\mathfrak{M}^{\prime}},$ etc. and hope no confusion will arise. Also,
the Breuil-Kisin module $\left(  \mathfrak{M},\phi\right)  $ is usually
denoted by $\mathfrak{M}.$ Finally, any unadorned tensor will be tensored over
$\mathfrak{S}$ in the usual way:\ we will continue to use $\otimes_{\phi}$
where appropriate.

Define the subcategory (Mod FI/$\mathfrak{S}$) to be the objects in $^{\prime
}$(Mod/$\mathfrak{S}$) which are isomorphic as $\mathfrak{S}$-modules to a
finite sum $\bigoplus\mathfrak{S}/p^{n_{i}}\mathfrak{S}$ for some choice of
$n_{i}$'s. Also, we let (Mod/$\mathfrak{S}$) be the subcategory of $^{\prime}%
$(Mod/$\mathfrak{S}$) consisting of modules with projective dimension 1 (as an
$\mathfrak{S}$-module). Alternatively, (Mod/$\mathfrak{S}$) is the subcategory
consisting of extensions in $^{\prime}$(Mod/$\mathfrak{S}$) of finite free
$\mathfrak{S}/p\mathfrak{S}$-modules \cite[2.3.2]{Kisin06}. Note that any
object in (Mod FI/$\mathfrak{S}$) is necessarily an object of
(Mod/$\mathfrak{S}$). We shall refer to objects in (Mod/$\mathfrak{S}$) as
Breuil-Kisin modules, with the understanding that the objects in (Mod
FI/$\mathfrak{S}$) are also Breuil-Kisin modules. From \cite[2.3.5 and
2.3.6]{Kisin06} we have anti-equivalences between (Mod/$\mathfrak{S}$) and the
category of group schemes over $R;$ and between (Mod FI/$\mathfrak{S}$) and
group schemes over $R$ with the property that the kernel of the map $\left[
p^{i}\right]  $ is also finite flat for all $i$.

The proofs of the above categorical equivalences rely on Breuil modules. We
briefly describe Breuil modules -- see \cite{Breuil00} for a more thorough
treatment, or see the summary in \cite{Kisin09}. Let
\[
S=\left\{  \sum_{i=0}^{\infty}w_{i}\frac{u^{i}}{\left\lfloor i/e\right\rfloor
!}\,|\,w_{i}\in W,\;\lim_{i\rightarrow\infty}w_{i}=0\right\}  \subset
K_{0}\left[  \left[  u\right]  \right]
\]
and let Fil$^{1}S$ be the $p$-adic completion of the ideal generated by
$\left(  E\left(  u\right)  \right)  ^{i}/i!.$ The ring $S$ is equipped with a
semilinear map $\phi$ given by%
\[
\phi\left(  \frac{u^{i}}{\left\lfloor i/e\right\rfloor !}\right)
=\frac{u^{pi}}{\left\lfloor i/e\right\rfloor !}\in pS
\]
and we define $\phi_{1}$ on $S$ to be ``$\phi/p$''. A Breuil module consists
of a triple $\left(  \mathcal{M},\mathcal{M}_{1},\phi_{1}\right)  $ where
$\mathcal{M}$ is a finite free module over $S,\;\mathcal{M}_{1}$ is an
$S$-submodule of $\mathcal{M}$ containing $\left(  \text{Fil}^{1}S\right)
\mathcal{M},$ and $\phi_{1}$ is a semilinear map $\mathcal{M}_{1}%
\mathcal{\rightarrow M}$ whose image generates $\mathcal{M}$ as an $S$-module.
The morphisms are $S$-module maps which commute with the respective $\phi_{1}%
$'s. The category of Breuil modules is denoted $^{\prime}$(Mod/$S$). If,
furthermore, the Breuil module $\mathcal{M}$ is isomorphic as an $S$-module to
a (finite) direct sum of $S/p^{n_{i}}S$ for various $n_{i}$'s then we say
$\mathcal{M}$ is an object in the subcategory (Mod FI/$S$). Also, in a manner
analogous to the Breuil-Kisin categories we let (Mod/$S$) be the full
subcategory of $^{\prime}$(Mod/$S$) which contains the objects in (Mod FI/$S$)
killed by $p$ and is stable by extensions.

Given a Breuil-Kisin module $\mathfrak{M}$ one can obtain a Breuil module as
follows. Let $\mathcal{M}=S\otimes_{\phi}\mathfrak{M},$ where $S$ is viewed as
an $\mathfrak{S}$-module via the map $\mathfrak{S}\rightarrow S,\;u\mapsto u.$
Then define%
\begin{align*}
\mathcal{M}_{1}  &  =\left\{  y\in\mathcal{M\,}|\,\left(  1\otimes\phi\right)
\left(  y\right)  \in\text{Fil}^{1}S\otimes_{\mathfrak{S}}\mathfrak{M}\right\}
\\
\phi_{1}  &  =\left(  \phi_{1}\otimes_{\phi}1\right)  \left(  1\otimes
\phi\right)  :\mathcal{M}_{1}\rightarrow\text{Fil}^{1}S\otimes\mathfrak{M}%
\rightarrow S\otimes_{\phi}\mathfrak{M}%
\end{align*}
Then $\left(  \mathcal{M},\mathcal{M}_{1},\phi_{1}\right)  $ is a Breuil
module, and this assignment is an exact, fully faithful functor
(Mod/$\mathfrak{S}$)$\rightarrow$(Mod/$S$) which restricts to (Mod
FI/$\mathfrak{S}$)$\rightarrow$(Mod FI/$S$). Thus one obtains a functor
Gr:\ (Mod/$\mathfrak{S}$)$\rightarrow$($p$-Gr), where $\left(  p\text{-Gr}%
\right)  $ denotes the category of finite flat group schemes of $p$-power order.

It is a consequence of \cite[2.4.7]{Kisin05} that two objects $\mathfrak{M}$,
$\mathfrak{M}^{\prime}$ in (Mod/$\mathfrak{S}$) correspond to group schemes
with isomorphic generic fibers if and only if $\mathfrak{M}\left[
u^{-1}\right]  \cong\mathfrak{M}^{\prime}\left[  u^{-1}\right]  .$ This
observation is very important for our purposes.

We can translate the anti-equivalences with group schemes to equivalences with
Hopf orders. The category (Mod/$\mathfrak{S}$) corresponds to the category of
$R$-Hopf algebras, and (Mod FI/$\mathfrak{S}$) is equivalent to the category
of Hopf algebras where $H/\left[  p^{i}\right]  H$ is flat for all $i$.
Furthermore, if $\mathfrak{M}\left[  u^{-1}\right]  \cong\mathfrak{M}^{\prime
}\left[  u^{-1}\right]  $ then the Hopf algebras corresponding to
$\mathfrak{M}$ and $\mathfrak{M}^{\prime}$ are generically isomorphic, hence
they are orders in the same $K$-Hopf algebra.

It should be pointed out that there is also a theory of Breuil-Kisin modules
for $p$-divisible groups. In this case $\mathfrak{M}$ is a free $\mathfrak{S}%
$-module, and so we have a collection of objects we could call free
``Breuil-Kisin modules''. We will not use these modules here, however they are
vital in the correspondence involving (Mod/$\mathfrak{S}$): the condition
``projective dimension 1'' means that we can realize $\mathfrak{M}$ as the
cokernel of free Breuil-Kisin modules, and the resulting sequence of modules
corresponds to a smooth resolution of the group scheme associated to
$\mathfrak{M}.$

\section{Cyclic Breuil-Kisin Modules}

The simplest type of objects in (Mod\ FI/$\mathfrak{S}$), and hence the
simplest type in (Mod/$\mathfrak{S}$), are the modules $\mathfrak{M}$ which
are generated over $\mathfrak{S}$ by a single element, say $e_{1}.$ Such a
module will be written $\mathfrak{S}_{n}e_{1}.\;$We will refer to such modules
as \emph{cyclic Breuil-Kisin modules}. In this section we will find all cyclic
Breuil-Kisin modules. These modules all have a simple form, particularly when
$n\neq1.$

Before we begin, we will prove a result on power series which will facilitate
calculations throughout the paper. For any ring $T$ we will denote by
$v:T\left[  \left[  u\right]  \right]  \rightarrow\mathbb{Z}^{+}$ the function
given by%
\[
v\left(  f\right)  =\min\left\{  v\,|\,f\in u^{v}T\left[  \left[  u\right]
\right]  \right\}  ,
\]
which is the usual $u$-adic valuation when $T=k.$ We will extend this to
$T\left(  \left(  u\right)  \right)  \rightarrow\mathbb{Z}$ as well. For $w\in
W$ we will let $w^{\phi}$ denote the Frobenius.

In the case where $n=1$ we get the simpler:

\begin{lemma}
\label{pslemma}Let $f$,$h$ $\in k\left[  \left[  u\right]  \right]  $ be
nonzero polynomials$.$ Let $f_{\ell}$ and $h_{m}$ be the coefficients of the
terms in $f$ and $h$ respectively of lowest degree. Then there exists a $g\in
k\left[  \left[  u\right]  \right]  ,\;g\neq0$ such that
\[
fg=\phi\left(  g\right)  h
\]
if and only if $v\left(  f\right)  \geq v\left(  h\right)  ,\;v\left(
f\right)  \equiv v\left(  h\right)  \,\left(  \operatorname{mod}\,p-1\right)
,\;$and if
\[
f_{\ell}/h_{m}\in\left(  k^{\times}\right)  ^{p-1}.
\]
Furthermore, if $f$ and $h$ are invertible then so is $g$.
\end{lemma}

\begin{proof}
Since $v\left(  \phi\left(  g\right)  \right)  =pv\left(  g\right)  $ the
conditions on $v\left(  f\right)  $ and $v\left(  h\right)  $ are obviously
necessary. If we write $f=u^{v\left(  f\right)  }f^{\prime}$ and
$h=u^{v\left(  h\right)  }h^{\prime}$ for $f^{\prime},h^{\prime}\in k\left[
\left[  u\right]  \right]  ^{\times}$ we see that this equation is equivalent
to
\[
u^{v\left(  f\right)  -v\left(  h\right)  }f^{\prime}g=\phi\left(  g\right)
h^{\prime}%
\]
and thus we may assume that $v\left(  h\right)  =0.$

Write $f=\sum f_{i}u^{i},\;g=\sum g_{i}u^{i},$ and $h=\sum h_{i}u^{i}$ and let
$j=v\left(  g\right)  .$ Then $v\left(  f\right)  =j\left(  p-1\right)  .\;$
By comparing $u^{pj}$ coefficients we get%
\[
f_{j\left(  p-1\right)  }g_{j}=\left(  g_{j}^{\phi}\right)  h_{0}.
\]
This determines $g_{j}$ since $f_{j\left(  p-1\right)  }$ is invertible. Now
suppose $g_{j},g_{j+1},\dots,g_{j+i}$ have been chosen. If we compare
$u^{pj+i+1}$ coefficients we get%
\[
f_{j\left(  p-1\right)  }g_{j+i+1}+\cdots+f_{j\left(  p-1\right)  +i+1}%
g_{j}=\left(  g_{j}^{\phi}\right)  h_{i+1}+G
\]
where $G$ is an expression involving $h_{i},h_{i-1},\dots,h_{0}$ and
$g_{j},g_{j+1},\dots,g_{j+\varepsilon}$ where $\varepsilon$ is the largest
integer such that $p\left(  j+\varepsilon\right)  <j+i+1.$ Thus, $g_{j+i+1}$
is determined (and in fact is unique for a fixed $g_{j}$), and by induction
$fg=\phi\left(  g\right)  h$ has a solution. That the solution is invertible
follows by considering valuations.
\end{proof}

We are now ready to describe the cyclic Breuil-Kisin modules.

\begin{lemma}
For $n\geq1\ $Breuil-Kisin module structures on $\mathfrak{S}_{n}e_{1}$
correspond to factorizations of $E$ $\operatorname{mod}p^{n}.$
\end{lemma}

\begin{proof}
Any semilinear map on $\mathfrak{S}_{n}e_{1}$ is of the form $\phi_{f}\left(
e_{1}\right)  =fe_{1}$ for some $f\in W_{n}\left[  \left[  u\right]  \right]
.$ In order for $\left(  \mathfrak{S}_{n}e_{1},\phi_{f}\right)  $ to be a
Breuil-Kisin module it is necessary and sufficient that $Ee_{1}$ be in the
image of $\left(  1\otimes\phi_{f}\right)  \left(  \mathfrak{S\otimes}%
_{\phi_{f}}\mathfrak{S}_{n}e_{1}\right)  ;$ hence if $\left(  \mathfrak{S}%
e_{1},\phi_{f}\right)  $ is a Breuil-Kisin modules there exists an
$s\in\mathfrak{S}$ such that%
\[
Ee_{1}=\left(  1\otimes\phi_{f}\right)  \left(  s\otimes_{\phi}e_{1}\right)
=s\phi_{f}\left(  e_{1}\right)  =sfe_{1}%
\]
and so $E\equiv sf\,\left(  \operatorname{mod}p^{n}\right)  .$
\end{proof}

We start by looking at cyclic Breuil-Kisin modules where $n=1$. These
correspond to $R$-Hopf algebras of order $p$. One can immediately see the
parallels with the Tate-Oort classification \cite{TateOort70}.

\begin{proposition}
All Breuil-Kisin modules with $\mathfrak{M=S}_{1}e_{1}$ are of the form
$\left(  \mathfrak{S}_{1}e_{1},\phi_{bu^{r}}\right)  $ $\phi_{bu^{r}}$ for
$0\leq r\leq e$ and $b\in k^{\times}.$ Furthermore:

\begin{enumerate}
\item $\left(  \mathfrak{S}_{1}e_{1},\phi_{bu^{r}}\right)  \cong\left(
\mathfrak{S}_{1}e_{1},\phi_{b^{\prime}u^{r^{\prime}}}\right)  $ if and only if
$r=r^{\prime}$ and $b/b^{\prime}\in\left(  k^{\times}\right)  ^{p-1}.$

\item The\ Hopf algebra associated to $\left(  \mathfrak{S}_{1}e_{1}%
,\phi_{c_{0}^{-1}u^{e}}\right)  $ is $RC_{p}$ and the Hopf algebra associated
to $\left(  \mathfrak{S}_{1}e_{1},\phi_{1}\right)  $ is $\left(
RC_{p}\right)  ^{\ast}.$

\item The Hopf algebra associated to $\left(  \mathfrak{S}_{1}e_{1}%
,\phi_{bu^{r}}\right)  $ is an order in $KC_{p}$ if and only if $bc_{0}%
\in\left(  k^{\times}\right)  ^{p-1}$and $r\equiv e\,\left(
\operatorname{mod}\,p-1\right)  .$
\end{enumerate}
\end{proposition}

\begin{proof}
Let $\left(  \mathfrak{S}_{1}e_{1},\phi_{f}\right)  $ be a Breuil-Kisin
module. Write $f=u^{r}h,$ where $h\in k\left[  \left[  u\right]  \right]
^{\times}.$ As $f$ is a factor of $E$ it is clear that $r\leq e.$ Let $g\in
k\left[  \left[  u\right]  \right]  ^{\times}$ be a solution to $hg=b\phi
\left(  g\right)  $ where $b=h\left(  0\right)  \;$-- by Lemma \ref{pslemma}
such a $g$ exists. The map $\alpha:\left(  \mathfrak{S}_{1}e_{1},\phi_{u^{r}%
h}\right)  \rightarrow\left(  \mathfrak{S}_{1},e_{1},\phi_{bu^{r}}\right)  $
given by $\alpha\left(  e_{1}\right)  =ge_{1}$ satisfies%
\begin{align*}
\alpha\left(  \phi_{u^{r}h}\left(  e_{1}\right)  \right)   &  =\alpha\left(
u^{r}he_{1}\right)  =u^{r}ghe_{1}=u^{r}b\phi\left(  g\right)  e_{1}\\
&  =\phi\left(  g\right)  \phi_{bu^{r}}\left(  e_{1}\right)  =\phi_{bu^{r}%
}\left(  ge_{1}\right)  =\phi_{bu^{r}}\left(  \alpha\left(  e_{1}\right)
\right)
\end{align*}
and as $g\in k\left[  \left[  u\right]  \right]  ^{\times}$ we have an
isomorphism. Now suppose $\beta:\left(  \mathfrak{S}_{1}e_{1},\phi_{bu^{r}%
}\right)  \cong\left(  \mathfrak{S}_{1}e_{1},\phi_{b^{\prime}u^{r^{\prime}}%
}\right)  $ is an isomorphism, say $\beta\left(  e_{1}\right)  =ge_{1},\;g\in
k\left[  \left[  u\right]  \right]  ^{\times}$ (different from the $g$ above).
Then
\begin{align*}
\beta\phi_{bu^{r}}\left(  e_{1}\right)   &  =\beta\left(  bu^{r}e_{1}\right)
=bu^{r}ge_{1}\\
\phi_{b^{\prime}u^{r^{\prime}}}\beta\left(  e_{1}\right)   &  =\phi
_{b^{\prime}u^{r^{\prime}}}\left(  ge\right)  =\phi\left(  g\right)
b^{\prime}u^{r^{\prime}}e_{1}.
\end{align*}
Since $\beta$ commutes with the respective $\phi$'s we have $bu^{r}%
g=\phi\left(  g\right)  b^{\prime}u^{r^{\prime}}.$ As we must have equal
valuations, and since $v\left(  g\right)  =v\left(  \phi\left(  g\right)
\right)  =1,$ we have $r=r^{\prime}.$ Thus this equation reduces to
$bg=\phi\left(  g\right)  b^{\prime},$ which has a solution if and only if
$b/b^{\prime}\in\left(  k^{\times}\right)  ^{p-1}.$ This proves \textbf{1}.

To prove \textbf{2} we will find the corresponding Breuil module for $\left(
\mathfrak{S}_{1}e_{1},\phi_{bu^{r}}\right)  .$ Recall that $\mathcal{M}%
=S\otimes_{\phi}\mathfrak{S}_{1}e_{1}.$ As $pe_{1}=0$ and $\phi\left(
p\right)  =p$ we have $p\otimes_{\phi}e_{1}=0$ so we may replace $S$ with
$S_{1}:=S/pS,$ which we will identify with $k\left[  u\right]  /\left(
u^{pe}\right)  ,$ via the isomorphism in \cite[2.1.2.1]{Breuil00}, where this
ring is denoted $\tilde{S}_{1}$. Note that under this isomorphism $\phi
_{1}\left(  u^{e}\right)  =\phi_{1}\left(  E\right)  =c_{0}^{p}\in S_{1}.$
Thus $\mathcal{M}=\left\{  s\otimes_{\phi}e_{1}\,|\,s\in S_{1}\right\}  $ and
we have%
\[
\mathcal{M}_{1}=\left\{  s\otimes_{\phi}e_{1}\,|\,\left(  1\otimes\phi\right)
\left(  s\otimes_{\phi}e_{1}\right)  \in u^{e}S_{1}\otimes\mathfrak{S}%
e_{1}\right\}  .
\]
As $\left(  1\otimes\phi\right)  \left(  s\otimes_{\phi}e_{1}\right)
=su^{r}e_{1}$ we see that $\left(  1\otimes\phi\right)  \left(  s\otimes
_{\phi}e_{1}\right)  \in u^{e}S_{1}\otimes\mathfrak{S}e_{1}$ if and only if
$s\in u^{e-r}S_{1},$ hence
\[
\mathcal{M}=u^{e-r}S_{1}\otimes_{\phi}\mathfrak{S}e_{1}.
\]
Finally, we compute $\phi_{1}$:%
\begin{align*}
\phi_{1}\left(  u^{e-r}\otimes_{\phi}e_{1}\right)   &  =\left(  \phi
_{1}\otimes_{\phi}1\right)  \left(  1\otimes\phi\right)  \left(
u^{e-r}\otimes_{\phi}e_{1}\right) \\
&  =\left(  \phi_{1}\otimes_{\phi}1\right)  \left(  u^{e-r}\otimes bu^{r}%
e_{1}\right) \\
&  =\left(  \phi_{1}\otimes_{\phi}1\right)  \left(  bu^{e}\otimes e_{1}\right)
\\
&  =b^{p}c_{0}^{p}\otimes_{\phi}e_{1}.
\end{align*}
This is the Breuil module denoted $\mathcal{M}\left(  e-r,\left(
bc_{0}\right)  ^{p}\right)  $ in \cite{BreuilConradDiamondTaylor01}. If
$b=c_{0}^{-1}$ and $r=e$ we get $\mathcal{M}\left(  0,1\right)  ,$ whose
corresponding Hopf algebra is $RC_{p}$; if $b=1$ and $r=0$ we get
$\mathcal{M}\left(  e,c_{0}^{-p}\right)  $ which gives $\left(  RC_{p}\right)
^{\ast}$ \cite[5.2.1]{BreuilConradDiamondTaylor01}.

Finally, if $\left(  \mathfrak{M=S}_{1}e_{1},\phi_{bu^{r}}\right)  $
corresponds to a Hopf order in $KC_{p}$ then $\mathfrak{M}\left[
u^{-1}\right]  $ is isomorphic to $\left(  \mathfrak{S}_{1}e_{1}\left[
u^{-1}\right]  ,\phi_{c_{0}^{-1}u^{e}}\right)  .$ Suppose $\gamma
:\mathfrak{M}\left[  u^{-1}\right]  \rightarrow\left(  \mathfrak{S}_{1}%
e_{1}\left[  u^{-1}\right]  ,\phi_{c_{0}^{-1}u^{e}}\right)  $ is such an
isomorphism. Then, as $\mathfrak{S}_{1}$-modules, each is isomorphic to the
ring of Laurent series $k\left(  \left(  u\right)  \right)  .\;$We have
$\gamma\left(  e_{1}\right)  =ge_{1}$ for some $g_{1}\in k\left(  \left(
u\right)  \right)  ^{\times}$ and $\gamma\left(  \phi_{bu^{r}}\left(
e_{1}\right)  \right)  =\phi_{c_{0}^{-1}u^{e}}\left(  \gamma\left(
e_{1}\right)  \right)  .$ Thus,%
\[
bu^{r}ge_{1}=\phi\left(  g\right)  c_{0}^{-1}u^{e}e_{1},
\]
and since $v\left(  \phi\left(  g\right)  \right)  =pv\left(  g\right)  $ we
have%
\[
r+v\left(  g\right)  =pv\left(  g\right)  +e
\]
which implies $r-e=\left(  p-1\right)  v\left(  g\right)  $ so $\left(
p-1\right)  $ divides $e-r.$ Writing $g=u^{v}g^{\prime},\;g^{\prime}\in
k\left[  \left[  u\right]  \right]  ^{\times}$ gives
\[
bu^{r+v}g^{\prime}=u^{e+pv}\phi\left(  g^{\prime}\right)  c_{0}^{-1}%
\]
and since $r+v=e+pv$ (note that $v<0$) we get%
\[
bg^{\prime}=\phi\left(  g^{\prime}\right)  c_{0}^{-1},
\]
which has a solution if and only if $bc_{0}\in\left(  k^{\times}\right)
^{p-1}.$
\end{proof}

\begin{remark}
More generally, one can show that $\left(  \mathfrak{S}_{1}e_{1},\phi_{bu^{r}%
}\right)  $ and $\left(  \mathfrak{S}_{1}e_{1},\phi_{b^{\prime}u^{r^{\prime}}%
}\right)  $ correspond to generically isomorphic Hopf algebras if and only if
$r\equiv r^{\prime}\,\left(  \operatorname{mod}\,p-1\right)  $ and
$b/b^{\prime}\in k^{\times}.$ Also, the first Hopf algebra is contained in the
second if and only if $r\geq r^{\prime}.$
\end{remark}

Of course, if $\left(  \mathfrak{S}_{1}e_{1},\phi_{bu^{r}}\right)  $
corresponds to a Hopf order in $KC_{p}$ we may replace $b$ with $c_{0}^{-1}$
since they give isomorphic Breuil-Kisin modules. Thus:

\begin{corollary}
The Hopf orders in $KC_{p}$ correspond to Breuil-Kisin modules of the form
$\left(  \mathfrak{S}_{1}e_{1},\phi_{c_{0}^{-1}u^{r}}\right)  $ where $r\equiv
e\,\left(  \operatorname{mod}p-1\right)  .$ If we let $j=\left(  e-r\right)
/\left(  p-1\right)  ,$ one can realize this Hopf algebra as the Larson order
\[
R\left[  \frac{\sigma-1}{\pi^{j}}\right]  \subset K\left\langle \sigma
\right\rangle
\]
where $\sigma$ is a generator of $C_{p}.$ (See \cite{Larson76} for a
description of Larson orders.)
\end{corollary}

We now turn our attention to the case where $n\geq2.$ We will see that the
cyclic Breuil-Kisin modules fail to give many Hopf orders.

\begin{proposition}
Suppose $n\geq2.$ Then the semilinear maps on $\mathfrak{S}_{n}e_{1}$ which
give a Breuil-Kisin module structure are of the form $\phi_{b}$ or $\phi
_{bE},$ where $b$ is an invertible element in$\;W_{n}.$ Furthermore:

\begin{enumerate}
\item Gr$\left(  \left(  \mathfrak{S}_{n}e_{1},\phi_{b}\right)  \right)  $ is
of multiplicative type and Gr$\left(  \left(  \mathfrak{S}_{n}e_{1},\phi
_{bE}\right)  \right)  $ is \'{e}tale.

\item We have $\left(  \mathfrak{S}_{n}e_{1},\phi_{b}\right)  \cong\left(
\mathfrak{S}_{n}e_{1},\phi_{b^{\prime}}\right)  $ (resp. $\left(
\mathfrak{S}_{n}e_{1},\phi_{bE}\right)  \cong\left(  \mathfrak{S}_{n}%
e_{1},\phi_{b^{\prime}E}\right)  $) if and only if $b/b^{\prime}\in$ $\left(
W_{n}^{\times}\right)  ^{p-1}.$
\end{enumerate}
\end{proposition}

\begin{proof}
As $E$ is irreducible $\operatorname{mod}\,p^{n}$, if we write $E=fs$ then
either $v\left(  f\right)  =0$ or $e$. In the first case, the map
$e_{1}\mapsto ge_{1},$ where $g\in W_{n}\left[  \left[  u\right]  \right]  $
is chosen so that $fg=\phi\left(  g\right)  b,\;$where $b=f\left(  0\right)
\in W_{n}^{\times},$ establishes an isomorphism $\left(  \mathfrak{S}_{n}%
e_{1},\phi_{f}\right)  \rightarrow\left(  \mathfrak{S}_{n}e_{1},\phi
_{b}\right)  .$ This can be shown by setting the constant term of $g$ equal to
1 and proceeding inductively as in the proof of Lemma \ref{pslemma}. If
$v\left(  f\right)  =e$ then $v\left(  s\right)  =0$ and we have $f=s^{-1}E.$
In this case choose $g\in W_{n}\left[  \left[  u\right]  \right]  $ such that
$gs^{-1}=b\phi\left(  g\right)  $ where $b=s^{-1}\left(  0\right)  \in
W_{n}^{\times}.$ This gives an isomorphism $\left(  \mathfrak{S}_{n}e_{1}%
,\phi_{f}\right)  \rightarrow\left(  \mathfrak{S}_{n}e_{1},\phi_{bE}\right)  .$

Statement \textbf{1} follows from \cite[1.1.15]{Kisin09}. For \textbf{2}, let
$\alpha:\left(  \mathfrak{S}_{n}e_{1},\phi_{b}\right)  \rightarrow\left(
\mathfrak{S}_{n}e_{1},\phi_{b^{\prime}}\right)  $ be the isomorphism given by
$e_{1}\mapsto he_{1}.$ Then $bhe_{1}=\phi\left(  h\right)  b^{\prime}e_{1},$
and as $b,b^{\prime}$ are invertible equality holds if and only if
$b/b^{\prime}=h_{0}^{\phi}/h_{0}$ where $h_{0}$ is the constant term of $h$. A
similar argument holds for the modules corresponding to the \'{e}tale groups.
\end{proof}

\begin{corollary}
Gr$\left(  \mathfrak{S}_{n}e_{1},\phi_{c_{0}^{-1}E}\right)  \cong\mu_{p^{n}}$
and Gr$\left(  \mathfrak{S}_{n}e_{1},\phi_{1}\right)  \cong\mathbb{Z}%
/p^{n}\mathbb{Z},$ and hence these Breuil modules correspond to $RC_{p^{n}}$
and $\left(  RC_{p^{n}}\right)  ^{\ast}$ respectively.
\end{corollary}

\begin{example}
Suppose $K=\mathbb{Q}_{p}\left[  \zeta\right]  $ where $\zeta$ is a primitive
$\left(  \zeta^{p^{n}}\right)  ^{\text{th}}$ root of unity. Then
\[
E=\frac{\left(  u+1\right)  ^{p^{n}}-1}{\left(  u+1\right)  ^{p^{n-1}}%
-1}=\frac{t^{p}-1}{t-1},\;t=\left(  u+1\right)  ^{p^{n-1}}%
\]
is its Eisenstein polynomial \cite[Lemma 3]{Birch67}. In particular,
$c_{0}=1.$ We have a map $\alpha:\left(  \mathfrak{S}_{n}e_{1}\left[
u^{-1}\right]  ,\phi_{1}\right)  \rightarrow\left(  \mathfrak{S}_{n}%
e_{1}\left[  u^{-1}\right]  ,\phi_{E}\right)  $ given by $\alpha\left(
e_{1}\right)  =\left(  1-t\right)  ^{-1}e_{1}.$ Notice that $\phi\left(
\left(  1-t\right)  ^{-1}\right)  =\left(  1-t^{p}\right)  ^{-1}.\;$Since%
\begin{align*}
\phi_{E}\alpha\left(  e_{1}\right)   &  =\phi_{E}\left(  \frac{1}{1-t}%
e_{1}\right) \\
&  =\left(  \frac{t^{p}-1}{t-1}\right)  \frac{1}{1-t^{p}}e_{1}\\
&  =\frac{1}{1-t}e_{1}=\alpha\phi\left(  e_{1}\right)
\end{align*}
we see that $\alpha$ is an isomorphism. This demonstrates the well-known fact
that $\mu_{p^{n}}\cong\mathbb{Z}/p^{n}\mathbb{Z}$ over a field $K$ containing
the $\left(  p^{n}\right)  ^{\text{th}}$ roots of unity.
\end{example}

The final result of the section shows the paucity of Hopf orders in group
rings arising from cyclic modules.

\begin{corollary}
The Hopf algebra associated to $\left(  \mathfrak{S}_{n}e_{1},\phi
_{bE}\right)  $ is a Hopf order in $KC_{p^{n}}$ if and only if it is
$RC_{p^{n}}.$
\end{corollary}

\begin{proof}
Clear since every isomorphism $\left(  \mathfrak{S}_{n}e_{1}\left[
u^{-1}\right]  ,\phi_{bE}\right)  \rightarrow\left(  \mathfrak{S}_{n}%
e_{1}\left[  u^{-1}\right]  ,\phi_{c_{0}^{-1}E}\right)  $ of the form
$e_{1}\mapsto he_{1}$ has $v\left(  h\right)  =0$ and hence restricts to an
isomorphism $\left(  \mathfrak{S}_{n}e_{1},\phi_{bE}\right)  \rightarrow
\left(  \mathfrak{S}_{n}e_{1},\phi_{c_{0}^{-1}E}\right)  .$
\end{proof}

\begin{remark}
Of course, a similar statement holds for orders in $\left(  KC_{p^{n}}\right)
^{\ast}$ or any other $K$-Hopf algebra of dimension $p^{n},\;n>2$ which is
realizable as $H\otimes_{R}K$ for some $R$-Hopf algebra $H$ corresponding to a
cyclic Breuil-Kisin module.
\end{remark}

\section{Hopf Orders in $KC_{p^{2}}$}

We now find the Breuil-Kisin modules corresponding to Hopf orders in
$KC_{p^{2}}.\;$The technique presented below is similar to the calculation in
\cite{Caruso09} of group schemes generically isomorphic to (``models of'')
$\mathbb{Z}/p^{2}\mathbb{Z}.$ Notice that all of the cyclic Breuil-Kisin
modules are objects in (Mod FI/$\mathfrak{S}$). The non-cyclic ones
constructed here will not be in this category, but as they are constructed
from extensions of objects in (Mod FI/$\mathfrak{S}$) they are in
(Mod/$\mathfrak{S}$).

\begin{theorem}
\label{p2}Let $0\leq j_{2}<j_{1}\leq e/\left(  p-1\right)  $ and pick $f\in
k\left(  \left(  u\right)  \right)  $ such that
\begin{align*}
v\left(  u^{e+j_{1}}\phi\left(  f\right)  -u^{e+j_{1}-\left(  p-1\right)
j_{2}}f\right)   &  \geq e-\left(  p-1\right)  \left(  j_{1}+j_{2}\right)  \\
v\left(  u^{j_{1}-pj_{2}}F+\left(  u^{e+j_{1}}\phi\left(  f\right)
-u^{e+j_{1}-j_{2}\left(  p-1\right)  }f\right)  \right)   &  \geq0.
\end{align*}
Let $\mathfrak{M=S}_{2}e_{1}+\mathfrak{S}_{2}e_{2}$ with $pe_{2}%
=u^{j_{1}-j_{2}}e_{1}.$ Let $\phi$ be the semilinear map on $\mathfrak{M}$
given by%
\begin{align*}
\phi\left(  e_{1}\right)   &  =c_{0}^{-1}u^{e-\left(  p-1\right)  j_{1}}%
e_{1}\\
\phi\left(  e_{2}\right)   &  =u^{-\left(  p-1\right)  j_{2}}c_{0}^{-1}%
Ee_{2}+\left(  u^{e+j_{1}}\phi\left(  f\right)  -u^{e+j_{1}-\left(
p-1\right)  j_{2}}f\right)  c_{0}^{-1}e_{1}.
\end{align*}
Then $\mathfrak{M}$ is a Breuil-Kisin module, and the corresponding group
scheme has generic fiber $\mu_{p^{2}}.$ Conversely, any such group scheme
isomorphic to $\mu_{p^{2}}$ over $K$ but not over $R$ has a Breuil-Kisin
module of the above form.
\end{theorem}

\begin{proof}
For the most part, it is easy to check that the module above produces a
Breuil-Kisin module. To show that $Ee_{2}$ is in the image of $1\otimes\phi$
we can write%
\[
Ee_{2}=c_{0}b^{\prime}u^{\left(  p-1\right)  j_{2}}\phi\left(  e_{1}\right)
+c_{0}u^{\left(  p-1\right)  j_{2}}\phi\left(  e_{2}\right)
\]
where $u^{e+j_{1}}\phi\left(  f\right)  -u^{e+j_{1}-\left(  p-1\right)  j_{2}%
}f=u^{e-\left(  p-1\right)  j_{1}}b^{\prime}.$ The map $e_{2}\mapsto
u^{-j_{2}}+pf$ (and hence $e_{1}\mapsto pu^{-j_{1}}$) establishes the
isomorphism $\mathfrak{M}\left[  u^{-1}\right]  \mathfrak{\rightarrow}%
W_{2}\left(  \left(  u\right)  \right)  $ that we need. The remainder of the
proof will establish that the conditions above are necessary.

Let $\mathfrak{M}$ be a Breuil-Kisin module over $W_{2}\left[  \left[
u\right]  \right]  $ with $\mathfrak{M}\left[  u^{-1}\right]  \cong
W_{2}\left(  \left(  u\right)  \right)  $, where $\phi_{0}$ on $W_{2}\left(
\left(  u\right)  \right)  $ is the semilinear map given by $\phi_{0}\left(
1\right)  =c_{0}^{-1}E.$ (We use the notation $\phi_{0}$ to eliminate
confusion since $\phi_{0}\left(  f\right)  =\phi\left(  f\right)  c_{0}^{-1}E$
for all $f\in W_{2}\left(  \left(  u\right)  \right)  .$) Let $\mathfrak{M}%
_{1}=\ker\left\{  p:\mathfrak{M\rightarrow M}\right\}  $ and let
$\mathfrak{M}_{2}=\mathfrak{M/M}_{1}.$ Then $p\mathfrak{M}_{1}=0=p\mathfrak{M}%
_{2}.$ Since $\phi\left(  pm\right)  =p\phi\left(  m\right)  ,$ $m\in
\mathfrak{M}$ we have that $\mathfrak{M}_{1}$ and $\mathfrak{M}_{2}$ are each
Breuil-Kisin modules (via $\phi|_{\mathfrak{M}_{1}}$ and $\bar{\phi}$
respectively). The isomorphism $\mathfrak{M}\left[  u^{-1}\right]  \cong
W_{2}\left(  \left(  u\right)  \right)  $ carries $p\mathfrak{M}\left[
u^{-1}\right]  $ to $pW_{2}\left(  \left(  u\right)  \right)  ,$ hence
$\mathfrak{M}_{1}\left[  u^{-1}\right]  $ and $\mathfrak{M}_{2}\left[
u^{-1}\right]  $ are each isomorphic to $k\left(  \left(  u\right)  \right)
.$ Thus there exist $e_{1}\in\mathfrak{M}_{1}$ and $\bar{e}_{2}\in
\mathfrak{M}_{2}$ such that
\begin{align*}
\mathfrak{M}_{1}  &  =\mathfrak{S}_{1}e_{1},\phi\left(  e_{1}\right)
=c_{0}^{-1}u^{e-\left(  p-1\right)  j_{1}}e_{1},\,0\leq j_{1}\leq\frac{e}%
{p-1}\\
\mathfrak{M}_{2}  &  =\mathfrak{S}_{1}\bar{e}_{2},\phi\left(  \bar{e}_{2}%
\right)  =c_{0}^{-1}u^{e-\left(  p-1\right)  j_{2}}\bar{e}_{2},\,0\leq
j_{2}\leq\frac{e}{p-1}.
\end{align*}
Pick $e_{2}\in\mathfrak{M}$ a lift of $\bar{e}_{2}.$ Then $\left\{
e_{1},e_{2}\right\}  $ generate $\mathfrak{M}\left[  u^{-1}\right]  $ as a
$W_{2}\left(  \left(  u\right)  \right)  $-module. Since $pe_{2}%
\in\mathfrak{M}_{1}$ it follows that
\[
pe_{2}=u^{\varepsilon}fe_{1}%
\]
for some $\,f\in k\left[  \left[  u\right]  \right]  ^{\times}\;$and
$\varepsilon\geq0.$ In fact, if $\varepsilon=0$ then $\mathfrak{M}$ is a
cyclic Breuil-Kisin module, and hence the corresponding group scheme is
isomorphic to $\mu_{p^{2}}$, thus we assume $\varepsilon>0.$ Applying $\phi$
to both sides gives us%
\[
p\phi\left(  e_{2}\right)  =u^{p\varepsilon}\phi\left(  f\right)  \phi\left(
e_{1}\right)
\]
and since $\phi\left(  e_{2}\right)  =c_{0}^{-1}u^{e-\left(  p-1\right)
j_{2}}e_{2}+pm$ for some $m\in\mathfrak{M}$ we get%
\begin{align*}
pc_{0}^{-1}u^{e-\left(  p-1\right)  j_{2}}e_{2}  &  =u^{p\varepsilon}%
\phi\left(  f\right)  c_{0}^{-1}u^{e-\left(  p-1\right)  j_{1}}e_{1}\\
&  =u^{\varepsilon\left(  p-1\right)  }\phi\left(  f\right)  c_{0}%
^{-1}u^{e-\left(  p-1\right)  j_{1}}f^{-1}\left(  u^{\varepsilon}fe_{1}\right)
\\
&  =u^{\varepsilon\left(  p-1\right)  }\phi\left(  f\right)  c_{0}%
^{-1}u^{e-\left(  p-1\right)  j_{1}}f^{-1}pe_{2}%
\end{align*}
and by comparing valuations we get%
\[
e-\left(  p-1\right)  j_{2}=\varepsilon\left(  p-1\right)  +e-\left(
p-1\right)  j_{1},
\]
i.e. $\varepsilon=j_{1}-j_{2}.$ Furthermore we see that $\phi\left(  f\right)
=f,$ so $f\in\mathbb{F}_{p}\left[  \left[  u^{p}\right]  \right]  ^{\times}.$
By replacing $e_{1}$ with $f^{-1}e_{1}$ we may assume $f=1.$ Thus%
\[
pe_{2}=u^{j_{1}-j_{2}}e_{1},
\]
and $j_{1}>j_{2}.$

To determine $\phi\left(  e_{2}\right)  $ we use the isomorphism
$\alpha:\mathfrak{M}\left[  u^{-1}\right]  \rightarrow W_{2}\left(  \left(
u\right)  \right)  $ which commutes with the $\phi$'s$.$ Let $\alpha\left(
e_{2}\right)  =u^{i}g,\;g\in W_{2}\left[  \left[  u\right]  \right]  ^{\times
}.$ By replacing $e_{2}$ by $\zeta e_{2},$ for some $\left(  p-1\right)
^{\text{st}}$ root of unity $\zeta$ we may assume $g\left(  0\right)
\equiv1\,\left(  \operatorname{mod}p\right)  .$ Then%
\[
\alpha\left(  \phi\left(  e_{2}\right)  \right)  =\phi_{0}\left(
\alpha\left(  e_{2}\right)  \right)  =u^{pi}\phi\left(  g\right)  c_{0}^{-1}E
\]
Since $\phi\left(  e_{2}\right)  \equiv c_{0}^{-1}u^{e-\left(  p-1\right)
j_{2}}e_{2}\,\left(  \operatorname{mod}\,p\right)  $ we have%
\begin{align*}
\alpha\left(  c_{0}^{-1}u^{e-\left(  p-1\right)  j_{2}}e_{2}+pm^{\prime
}\right)   &  =c_{0}^{-1}u^{e-\left(  p-1\right)  j_{2}}\alpha\left(
e_{2}\right)  +p\alpha\left(  m^{\prime}\right) \\
&  =u^{pi}\phi\left(  g\right)  c_{0}^{-1}E,
\end{align*}
for some $m^{\prime}\in\mathfrak{M}$, and so
\[
\alpha\left(  e_{2}\right)  \equiv u^{pi+\left(  p-1\right)  j_{2}-e}%
\phi\left(  g\right)  E\,\left(  \operatorname{mod}\,p\right)  .
\]
Since $\phi\left(  g\right)  E\equiv u^{e}\,\left(  \operatorname{mod}%
\,p\right)  $ we get%
\[
u^{i}\equiv u^{pi+\left(  p-1\right)  j_{2}}\,\left(  \operatorname{mod}%
p\right)
\]
and hence $i=pi+\left(  p-1\right)  j_{2},\,$i.e. $i=-j_{2}.$ Therefore,%
\[
\alpha\left(  e_{2}\right)  =u^{-j_{2}}+pf,\;f\in k\left(  \left(  u\right)
\right)  .
\]
Since $\left(  u^{j_{2}}-u^{2j_{2}}pf\right)  \left(  u^{-j_{2}}+pf\right)
=1\in W_{2}\left(  \left(  u\right)  \right)  $ we get%
\begin{align*}
\alpha\left(  \phi\left(  e_{2}\right)  \right)   &  =\phi_{0}\left(
u^{-j_{2}}+pf\right) \\
&  =\left(  u^{-pj_{2}}+p\phi\left(  f\right)  \right)  c_{0}^{-1}E\\
&  =\left(  u^{-pj_{2}}+p\phi\left(  f\right)  \right)  c_{0}^{-1}E\left(
u^{j_{2}}-u^{2j_{2}}pf\right)  \left(  u^{-j_{2}}+pf\right) \\
&  =\left(  u^{-pj_{2}}+p\phi\left(  f\right)  \right)  \left(  u^{j_{2}%
}-u^{2j_{2}}pf\right)  c_{0}^{-1}E\alpha\left(  e_{2}\right) \\
&  =\alpha\left(  \left(  u^{-pj_{2}}+p\phi\left(  f\right)  \right)  \left(
u^{j_{2}}-u^{2j_{2}}pf\right)  c_{0}^{-1}Ee_{2}\right)  .
\end{align*}
As $\alpha$ is an isomorphism we get%
\begin{align*}
\phi\left(  e_{2}\right)   &  =\left(  u^{-pj_{2}}+p\phi\left(  f\right)
\right)  \left(  u^{j_{2}}-u^{2j_{2}}pf\right)  c_{0}^{-1}Ee_{2}\\
&  =\left(  u^{-j_{2}\left(  p-1\right)  }+pu^{j_{2}}\phi\left(  f\right)
-u^{-j_{2}\left(  p-2\right)  }pf\right)  c_{0}^{-1}Ee_{2}\\
&  =u^{-j_{2}\left(  p-1\right)  }c_{0}^{-1}Ee_{2}+\left(  u^{j_{2}}%
\phi\left(  f\right)  -u^{-j_{2}\left(  p-2\right)  }f\right)  c_{0}%
^{-1}E\left(  pe_{2}\right) \\
&  =u^{-j_{2}\left(  p-1\right)  }c_{0}^{-1}Ee_{2}+\left(  u^{e+j_{1}}%
\phi\left(  f\right)  -u^{e+j_{1}-j_{2}\left(  p-1\right)  }f\right)
c_{0}^{-1}e_{1}.
\end{align*}
Now it is necessary that the right-hand side be in $\mathfrak{M}$ (as opposed
to $\mathfrak{M}\left[  u^{-1}\right]  $), therefore there are restrictions on
the choice of $f$. We will return to this issue at the end of the proof.

Since $\mathfrak{M}$ to be a Breuil-Kisin module, we require that that
$Ee_{1}$ and $Ee_{2}$ are in the image of $1\otimes\phi.$ As $Ee_{1}%
=c_{0}u^{\left(  p-1\right)  j_{1}}\phi\left(  e_{1}\right)  $ it suffices to
find $x,y\in W_{2}\left[  \left[  u\right]  \right]  $ such that $Ee_{2}%
=x\phi\left(  e_{1}\right)  +y\phi\left(  e_{2}\right)  .$ Thus%
\begin{align*}
pEe_{2}  &  =yu^{-\left(  p-1\right)  j_{2}}c_{0}^{-1}Epe_{2}\\
pu^{e}e_{2}  &  =yu^{e-\left(  p-1\right)  j_{2}}c_{0}^{-1}pe_{2}%
\end{align*}
and hence $y=c_{0}u^{\left(  p-1\right)  j_{2}}+py^{\prime}$ for some
$y^{\prime}\in k\left[  \left[  u\right]  \right]  .$ Substituting, we get%
\begin{align*}
Ee_{2}  &  =c_{0}^{-1}xu^{e-\left(  p-1\right)  j_{1}}e_{1}+\left(
c_{0}u^{\left(  p-1\right)  j_{2}}+py^{\prime}\right)  \phi\left(
e_{2}\right) \\
&  =c_{0}^{-1}xu^{e-\left(  p-1\right)  j_{1}}e_{1}+Ee_{2}+c_{0}%
^{-1}py^{\prime}u^{e-\left(  p-1\right)  j_{2}}e_{2}+u^{\left(  p-1\right)
j_{2}}\left(  u^{e+j_{1}}\phi\left(  f\right)  -u^{e+j_{1}-\left(  p-1\right)
j_{2}}f\right)  e_{1}%
\end{align*}
and so%
\begin{align*}
c_{0}^{-1}xu^{e-\left(  p-1\right)  j_{1}}e_{1}+c_{0}^{-1}py^{\prime
}u^{e-\left(  p-1\right)  j_{2}}e_{2}+u^{\left(  p-1\right)  j_{2}}\left(
u^{e+j_{1}}\phi\left(  f\right)  -u^{e+j_{1}-\left(  p-1\right)  j_{2}%
}f\right)  e_{1}  &  =0\\
\left(  c_{0}^{-1}xu^{e-\left(  p-1\right)  j_{1}}+c_{0}^{-1}y^{\prime
}u^{e-pj_{2}+j_{1}}+u^{\left(  p-1\right)  j_{2}}\left(  u^{e+j_{1}}%
\phi\left(  f\right)  -u^{e+j_{1}-\left(  p-1\right)  j_{2}}f\right)  \right)
e_{1}  &  =0
\end{align*}
As $e-pj_{2}+j_{1}>e-\left(  p-1\right)  j_{1}$, for this to have a solution
we require $\left(  p-1\right)  j_{2}+v\left(  u^{e+j_{1}}\phi\left(
f\right)  -u^{e+j_{1}-\left(  p-1\right)  j_{2}}f\right)  =v\left(  x\right)
+v\left(  u^{e-\left(  p-1\right)  j_{1}}\right)  ,$ i.e. $v\left(
u^{e+j_{1}}\phi\left(  f\right)  -u^{e+j_{1}-\left(  p-1\right)  j_{2}%
}f\right)  \geq e-\left(  p-1\right)  \left(  j_{1}+j_{2}\right)  .$ If we
write $u^{e+j_{1}}\phi\left(  f\right)  -u^{e+j_{1}-\left(  p-1\right)  j_{2}%
}f=u^{e-\left(  p-1\right)  \left(  j_{1}+j_{2}\right)  }b$ then we can solve
the above by setting $x=c_{0}b$ and $y^{\prime}=0.$

We require $\phi\left(  e_{2}\right)  \in\mathfrak{M}$, which of course is
equivalent to having $c_{0}\phi\left(  e_{2}\right)  \in\mathfrak{M.}$ We have%
\begin{align*}
c_{0}\phi\left(  e_{2}\right)   &  =u^{-j_{2}\left(  p-1\right)  }%
Ee_{2}+\left(  u^{e+j_{1}}\phi\left(  f\right)  -u^{e+j_{1}-j_{2}\left(
p-1\right)  }f\right)  e_{1}\\
&  =u^{-j_{2}\left(  p-1\right)  }\left(  u^{e}+pF\right)  e_{2}+\left(
u^{e+j_{1}}\phi\left(  f\right)  -u^{e+j_{1}-j_{2}\left(  p-1\right)
}f\right)  e_{1}\\
&  =u^{e-j_{2}\left(  p-1\right)  }e_{2}+\left(  u^{-j_{2}\left(  p-1\right)
}u^{j_{1}-j_{2}}F+\left(  u^{e+j_{1}}\phi\left(  f\right)  -u^{e+j_{1}%
-j_{2}\left(  p-1\right)  }f\right)  \right)  e_{1}\\
&  =u^{e-j_{2}\left(  p-1\right)  }e_{2}+\left(  u^{j_{1}-pj_{2}}F+\left(
u^{e+j_{1}}\phi\left(  f\right)  -u^{e+j_{1}-j_{2}\left(  p-1\right)
}f\right)  \right)  e_{1}%
\end{align*}
and since $u^{e-j_{2}\left(  p-1\right)  }e_{2}\in\mathfrak{M}$ this means we
need
\[
v\left(  u^{j_{1}-pj_{2}}F+\left(  u^{e+j_{1}}\phi\left(  f\right)
-u^{e+j_{1}-j_{2}\left(  p-1\right)  }f\right)  \right)  \geq0
\]
as desired.
\end{proof}

\begin{remark}
The second valuation condition is a bit more difficult to work with because it
depends on the Eisenstein polynomial. However, suppose we pick $j_{1}<j_{2}$
such that $e-\left(  p-1\right)  \left(  j_{1}+j_{2}\right)  \geq0.$ Then for
$f\in k\left(  \left(  u\right)  \right)  $ so that $v\left(  u^{e+j_{1}}%
\phi\left(  f\right)  -u^{e+j_{1}-\left(  p-1\right)  j_{2}}f\right)  \geq
e-\left(  p-1\right)  \left(  j_{1}+j_{2}\right)  $ we see that $\left(
u^{e+j_{1}}\phi\left(  f\right)  -u^{e+j_{1}-\left(  p-1\right)  j_{2}%
}f\right)  e_{1}\in\mathfrak{M}$ and hence $\left(  u^{j_{1}-pj_{2}}F+\left(
u^{e+j_{1}}\phi\left(  f\right)  -u^{e+j_{1}-j_{2}\left(  p-1\right)
}f\right)  \right)  e_{1}\in\mathfrak{M}$ precisely when $u^{j_{1}-pj_{2}%
}e_{1}\in\mathfrak{M.}$ Thus if $e-\left(  p-1\right)  \left(  j_{1}%
+j_{2}\right)  \geq0$ then the second condition is satisfied if and only if
$j_{1}\geq pj_{2}.$
\end{remark}

\section{Laurent Series and Hopf Orders}

Since each Breuil-Kisin module $\mathfrak{M}$ which is an order in $KC_{p^{n}%
}$ must satisfy $\mathfrak{M}\left[  u^{-1}\right]  \cong\left(
\mathfrak{S}_{n}e_{1}\left[  u^{-1}\right]  ,\phi_{c_{0}^{-1}E}\right)  ,$ as
$\mathfrak{S}$-modules we have $\mathfrak{M}\left[  u^{-1}\right]  \cong
W_{n}\left(  \left(  u\right)  \right)  .$ Thus Hopf orders can be identified
by looking at certain Laurent series.

As an example, let us return to the case $n=1.$ We have found that the Hopf
orders correspond to Breuil-Kisin modules of the form $\left(  \mathfrak{S}%
_{1}e_{1},\phi_{c_{0}^{-1}u^{r}}\right)  $ with $e-r=j\left(  p-1\right)  $
for some $j$. By following the induced isomorphism $\mathfrak{M}\left[
u^{-1}\right]  \rightarrow k\left(  \left(  u\right)  \right)  ,$ which here
can be chosen to be $e_{1}\mapsto u^{-j}e_{1},$ we have $e_{1}$ corresponding
to the Laurent series $u^{-j}.$ Thus, $u^{-j}$ encodes all of the information
concerning this Hopf order. In other words, the set
\[
\left\{  u^{-j}\,|\,0\leq j\leq\left\lfloor \frac{e}{p-1}\right\rfloor
\right\}  \subset k\left(  \left(  u\right)  \right)  .
\]
parameterizes all Hopf orders in $KC_{p}.$

Similarly, the Breuil-Kisin module $\mathfrak{M}$ for a Hopf order in
$KC_{p^{2}}$ is generated by at most two elements $e_{1}$ and $e_{2}$ such
that there is an isomorphism $\alpha:\mathfrak{M}\left[  u^{-1}\right]
\rightarrow W_{2}\left(  \left(  u\right)  \right)  .$ From the work above we
see that $\alpha\left(  e_{2}\right)  =u^{-j_{2}}+pf$ and $\alpha\left(
e_{1}\right)  =\alpha\left(  u^{j_{2}-j_{1}}pe_{2}\right)  =pu^{-j_{1}}$,
where $v\left(  u^{e+j_{1}}\phi\left(  f\right)  -u^{e+j_{1}-\left(
p-1\right)  j_{2}}f\right)  \geq e-\left(  p-1\right)  \left(  j_{1}%
+j_{2}\right)  $ and $v\left(  u^{j_{1}-pj_{2}}F+\left(  u^{e+j_{1}}%
\phi\left(  f\right)  -u^{e+j_{1}-j_{2}\left(  p-1\right)  }f\right)  \right)
\geq0.$ Thus the set%
\[%
\begin{array}
[c]{c}%
\{\left(  pu^{-j_{1}},u^{-j_{2}}+pf\right)  \,|\,j_{1}\geq pj_{2},\;v\left(
u^{e+j_{1}}\phi\left(  f\right)  -u^{e+j_{1}-\left(  p-1\right)  j_{2}%
}f\right)  \geq e-\left(  p-1\right)  \left(  j_{1}+j_{2}\right)  ,\\
v\left(  u^{j_{1}-pj_{2}}F+\left(  u^{e+j_{1}}\phi\left(  f\right)
-u^{e+j_{1}-j_{2}\left(  p-1\right)  }f\right)  \right)  \}\subset\left(
k\left(  \left(  u\right)  \right)  \right)  ^{2}%
\end{array}
\]
parameterizes all Hopf orders in $KC_{p^{2}}.$ It seems possible that one
could find Hopf orders in $KC_{p^{n}}$ by picking $n$-tuples of Laurent series
satisfying certain properties.

To what extent can this idea be applied to $n>2$? Pick $f_{1},f_{2}%
,\dots,f_{n}\in W_{n}\left(  \left(  u\right)  \right)  $ such that
$f_{1}\notin pW_{n}\left(  \left(  u\right)  \right)  $ and%
\[
f_{i}=\sum_{i=1}^{n}s_{ij}c_{0}^{-1}\phi\left(  f_{j}\right)  ,\;s_{ij}\in
W_{n}\left[  \left[  u\right]  \right]  .
\]
Let $\mathfrak{M}\ $be the $\mathfrak{S}$-module generated by $\left\{
e_{1},e_{2},\dots,e_{n}\right\}  $ such that $\alpha:W_{n}\left(  \left(
u\right)  \right)  \rightarrow\mathfrak{M}\left[  u^{-1}\right]
,\;\alpha\left(  f_{i}\right)  =e_{i}\;$is an $\mathfrak{S}$-module
isomorphism. Define $\phi:\mathfrak{M\rightarrow M}$ by $\phi\left(
e_{i}\right)  =E\alpha\left(  c_{0}^{-1}\phi\left(  f_{i}\right)  \right)  .$
Since%
\begin{align*}
Ee_{i}  &  =E\alpha\left(  f_{i}\right)  =E\alpha\left(  \sum_{i=1}^{n}%
s_{ij}c_{0}^{-1}\phi\left(  f_{j}\right)  \right) \\
&  =\sum_{i=1}^{n}s_{ij}E\alpha\left(  c_{0}^{-1}\phi\left(  f_{i}\right)
\right) \\
&  =\sum_{i=1}^{n}s_{ij}\phi\left(  e_{j}\right)
\end{align*}
we see that $\left(  \mathfrak{M},\phi\right)  $ is a Breuil-Kisin module.
Furthermore, if we let $\phi_{0}$ be the semilinear map on $W_{n}\left(
\left(  u\right)  \right)  $ given by $\phi_{0}\left(  1\right)  =c_{0}^{-1}E$
we get%
\begin{align*}
\phi\left(  \alpha\left(  f_{i}\right)  \right)   &  =\phi\left(
e_{i}\right)  =E\alpha\left(  c_{0}^{-1}\phi\left(  f_{i}\right)  \right) \\
\alpha\left(  \phi_{0}\left(  f_{i}\right)  \right)   &  =\alpha\left(
c_{0}^{-1}E\phi\left(  f_{i}\right)  \right)  =E\alpha\left(  c_{0}^{-1}%
\phi\left(  f_{i}\right)  \right)  ,
\end{align*}
and so $\alpha\phi_{0}=\phi\alpha$ and we get:

\begin{proposition}
There is a correspondence%
\[
\left\{  \mathbf{f}\in\left(  W_{n}\left(  \left(  u\right)  \right)  \right)
^{n}\,|\,\mathbf{f}=A\mathbf{\Phi}\,\left(  \mathbf{f}\right)  ,\,A\in
M_{n}\left(  W_{n}\left[  \left[  u\right]  \right]  \right)  ,\;f_{1}%
\not \equiv 0\,\left(  \operatorname{mod}p\right)  \right\}  ,
\]
where $\mathbf{f}=\left(  f_{1},f_{2},\dots,f_{n}\right)  $ and $\mathbf{\Phi
}\left(  \mathbf{f}\right)  =\left(  \phi\left(  f_{1}\right)  ,\phi\left(
f_{2}\right)  ,\dots,\phi\left(  f_{n}\right)  \right)  ;$ and $R$-Hopf orders
in $KC_{p^{n}}.$
\end{proposition}

\begin{proof}
The only remaining detail is to show that the $\mathfrak{M}$ as constructed
above is an object in $\left(  \text{Mod/}\mathfrak{S}\right)  $ (as opposed
to simply an object in $^{\prime}\left(  \text{Mod/}\mathfrak{S}\right)  $).
This, however, follows from the fact that there is a canonical surjection
$\mathfrak{S}^{n}\rightarrow\mathfrak{M}$.
\end{proof}

While it would be convenient to have a theory of Hopf orders based solely on
the ring of Laurent series, there seem to be two obstacles to using this
proposition in practice. One, it seems difficult in general to find the
$f_{i}$'s that satisfy the above conditions. Perhaps it would be possible to
proceed inductively, in which case the matrix $A$ would be upper-triangular.
The other problem is that this is not a one-to-one correspondence: many
choices of $\mathbf{f}$ lead to isomorphic Breuil-Kisin modules. For example,
in the case $n=1$ the correspondence becomes%
\begin{align*}
\left\{  f\in k\left(  \left(  u\right)  \right)  ^{\times}\,|\,f=s\phi\left(
f\right)  ,\;v\left(  s\right)  \geq0\right\}   &  =\left\{  f=au^{j}%
+u^{j+1}f^{\prime}\in k\left(  \left(  u\right)  \right)  \,|\,a\in\left(
k^{\times}\right)  ^{p-1},\;j\geq pj\right\} \\
&  =\left\{  f=au^{j}+u^{j+1}f^{\prime}\in k\left(  \left(  u\right)  \right)
\,|\,a\in\left(  k^{\times}\right)  ^{p-1},\;j\leq0\right\}  ,
\end{align*}
and we see that many different choices of $f$ correspond to the same $R$-Hopf
order. While it is easy to determine which $f$ are equivalent in this sense
when $n=1$, it seems much more involved for larger $n$.

\bibliographystyle{commalg}
\bibliography{Koch}
\end{document}